\documentclass[12pt]{article} 
\usepackage{color}
\usepackage{framed}
\usepackage{comment}
\definecolor{shadecolor}{gray}{0.875}
\specialcomment{nota}{\begin{shaded}}{\end{shaded}}
\includecomment{nota}
\usepackage{a4wide}

\usepackage[utf8]{inputenc}
\usepackage{enumerate}
\usepackage{amsmath}
\usepackage{amsthm}
\usepackage{amsfonts}
\usepackage{euscript}
\usepackage{graphicx}
\usepackage[all]{xy}

\theoremstyle{plain}
\newtheorem{theorem}{Theorem}[section]
\newtheorem{lemma}[theorem]{Lemma}

\newtheorem{definition}[theorem]{Definition}
\newtheorem{assumption}[theorem]{Assumption}

\theoremstyle{remark}

\newtheorem{remark}[theorem]{Remark}

\numberwithin{equation}{section}

\DeclareMathSymbol{\R}{\mathalpha}{AMSb}{"52}

\newcommand{\hc} {{\mathcal{H}}}
\newcommand{\hi} {\left(\pi \hc_t\right)^{-1}}
\newcommand{\hif}{\left(\pi \hc_{\widehat{T}}\right)^{-1}}

\newcommand{\oc} {{\EuScript O}}
\newcommand{\fc} {{\widehat{\mathcal{F}}}}
\newcommand{\uc} {{\EuScript U}}

\newcommand{\up}  {{\widehat \upsilon}}
\newcommand{\xa}  {{\widehat{x}_0}}
\newcommand{\xf}  {{\widehat{x}_f}}

\newcommand{\la}  {{\widehat{\ell}_0}}
\newcommand{\lf}  {{\widehat{\ell}_f}}
\newcommand{\tf}  {{\widehat{T}}}

\newcommand{\bs} {\boldsymbol{s}}
\newcommand{\bo} {\boldsymbol{\omega}}
\newcommand{\bsi} {\boldsymbol{\sigma}}

\renewcommand\^{\widehat }

\def\wh{\widehat}

\usepackage{pspicture}

\usepackage{pstricks, pst-node, pst-text, pst-3d}

\usepackage{authblk}
\title{A Hamiltonian approach to sufficiency in optimal control with minimal regularity conditions: Part I }
\author[1]{G. Stefani }
\author[2]{ P. Zezza }
\affil[1]{\small Dipartimento di Matematica e Informatica, Via di Santa Marta 3 -  50139 FIRENZE - ITALIA}
\affil[2]{Dipartimento di Scienze per l'Economia e l'Industria, Via delle Pandette 9 - 50137 FIRENZE - ITALIA}
\providecommand{\keywords}[1]{\noindent \textbf{\textit{Keywords.}} #1}

\providecommand{\classification}[1]{\noindent \textbf{\textit{MSC Classification.}} #1}

\begin{document}

\maketitle

\abstract{\noindent In this paper we develop a Hamiltonian approach to sufficient conditions in optimal control problems. We extend the known conditions for $C^2$ maximised Hamiltonians into two directions: on the  one hand we explain the role of a super Hamiltonian (i.e. a Hamiltonian which is greater then or equal to the maximised one) on the other  we develop the theory under some minimal regularity assumptions. The results we present  enclose many known results and they can be used to tackle new problems.}\\
\ \\
\\
\classification{49K15, 49J30}

\section{Introduction}
This paper is the first part of a larger project whose goal is to describe how a Hamiltonian approach can be a key instrument in optimal control problems (OCPs). 

Hamiltonian methods have been used in OCPs to state sufficient conditions ensuring  the strong local optimality of a reference trajectory. The main feature of this approach is that it allows us to compare the costs of different trajectories by lifting them to the cotangent bundle and hence independently of the control values. The seminal idea goes back to 1879 when K. Weierstrass discovered a method which enables one to establish a strong minimum property for solutions of Euler-Lagrange's equations, i.e. for stationary curves. We can summarise Weierstrass's method as a combination of a local convexity assumption on the integrand and of an embedding of the given curve in a suitable field of non-intersecting stationary curves.  The extension to optimal control requires some efforts since one has to deal with the maximised Hamiltonian coming from the Pontryagin Maximum Principle (PMP). The PMP was introduced in 1956 by Lev Pontryagin and it is an extension to OCPs of the Euler-Lagrange's equation for the Calculus of Variations (CV). One of the main difficulty being that the maximised Hamiltonian in OCPs does not possess the same regularity properties as the one in the CV.

The goal of this project is to propose a unified Hamiltonian approach to sufficient optimality conditions for OCPs whose state evolves on a manifold $M.$ These results include some known ones, emhpasizing their common features, and will allow us to tackle new problems.

The leading ideas of the project are the following:

\begin{itemize}

\item To use the symplectic properties of the cotangent bundle to compare the costs of neighbouring admissible trajectories by lifting them to the cotangent bundle. 

\item To define in the cotangent bundle $T^*M$ a suitable Hamiltonian flow $\hc_t$  emanating from a horizontal Lagrangian submanifold $\Lambda.$ This flow is the one of the maximised Hamiltonian when this last is at least $C^2.$ 

\item To estimate the variation of the cost of admissible trajectories by the variation of a function of their final points and, if it is the case, their final times.

\item To obtain a suitable second order approximation (\emph{$2^{nd}$ variation}) in the form of a coordinate-free Linear-Quadratic (LQ) problem and to require its coercivity.

\item To show that $\hc_{t*}$ (the derivative of $\hc_t$ along the reference extremal) is, up to an isomorphism, the linear Hamiltonian flow associated to the LQ problem.

\item To use the coercivity of the  $2^{nd}$  variation  to substitute the manifold described by the transversality conditions, $\Lambda_0,$ by an horizontal one $\Lambda$. This can be obtained  by adding a penalty term which  reduces the problem to  a problem with free initial point and whose  $2^{nd}$  variation is still coercive,  see M. Hestenes \cite{MR0046590}.  This allows us to overcome a difficult point: in optimal control problems, for lack of controllability it may happen that the projection of the flow starting from $\Lambda_0$ at $t=0$ is not locally onto for a non--trivial time interval.
 
\item To deduce for a problem with free initial point and fixed final point that the projection on $M$ of $\hc_{t}$ emanating from $\Lambda$ is locally invertible so that we can go back to the first issue and we can compare the costs of neighbouring admissible trajectories by lifting them to the cotangent bundle.
\item To use again the coercivity of the  $2^{nd}$  variation to complete the proof for the general case. 
\end{itemize}

The project goes back to some initial papers \cite{MR1819736} and \cite{MR1658472} where the use of this kind of approach to OCPs for the case when the maximised Hamiltonian is at least $C^2$ did  begin.

In \cite{MR1819736} the authors studied a Bolza problem in $\R^n$ on a fixed time interval and with general end--points constraints. They assumed that the data are uniformly quasi$-C^2$, see Definition 1 therein, where the \emph{quasi} refers to the fact that the data are time-dependent, moreover, under the strengthened Legendre condition,  assuming that in a neighbourhood of the reference Pontryagin extremal the local maximum of the Hamiltonian coincides with the global maximum 
they  proved that the maximised Hamiltonian is quasi$-C^2.$

This kind of regularity allows us to look for second order conditions which will guarantee that the abstract theory can be applied to obtain strong local optimality.

When the state evolves on manifolds, the first issue is to prove an invariant version of second order conditions; this problem has been addressed in \cite{MR1658472} for a Mayer problem on a fixed time interval. This result can be obtained by pulling back the control problem to a neighbourhood of the initial point $\xa,$ in this way the second order approximation leads naturally to derive a Hamiltonian formulation of the second variation as a linear-quadratic optimal control problem on $T_{\xa}M$ with control functions in $L^2.$ This formulation is coordinate-free and hence invariant.  

In the book \cite{MR2062547}, dedicated to the \emph{geometric} approach to control problems, the smooth case is investigated in a Hamiltonian setting for a Bolza problem with fixed end points; in \cite{MR2306634} the authors present a numerical algorithms to compute the first point where the trajectory ceases to be locally optimal. 

The assumption that the data are uniformly quasi$-C^2$ can be weakened as shown in \cite{MR1654537} by imposing conditions directly on the Hamiltonian flow (Theorem 2.3 therein), these conditions make sense even if the second variation does not exist and, when the second variation exists, they are implied by its coercivity.

Here we develop the first part of the project and we relax the regularity assumptions on the Hamiltonian in such a way that they guarantee the existence and the regularity of the flow so that this technique can be applied to a larger class of different OCPs.
The Hamiltonian we consider can be either the maximised Hamiltonian or a suitable super--Hamiltonian  which is not necessarily equal to the maximised one but which will be needed in the study of problems where the control contains a singular arc.

The regularity assumptions we propose allows us to use the symplectic properties of Hamiltonian flows (Lemma \ref{exact2})  to state abstract sufficient optimality conditions for a Bolza problem (Theorem \ref{main1}) and we show their possible applications to a minimum time problem (Theorem \ref{main2}).

In the final Section \ref{appl} we briefly summarise the optimal control problems where our assumptions are satisfied and this approach has been used.

In a forthcoming paper we will give a more detailed description of this approach and larger set of references, furthermore we will give some suggestions about possible extensions and further investigations.  For different approaches here we quote only  \cite{MR2970901}, \cite{MR1972537}, \cite{MR2081418} and reference therein. 

\subsection{The problem}
We consider the following optimisation  problem
\begin{center}
Minimise
\end{center}
\begin{equation}
J (T,\,\xi,\,\upsilon) = c_0(\xi (0))+c_f(\xi (T))+\int_{0}^{T}\!\! f^0\bigl(\xi (t),\upsilon(t)\bigr)\, dt \label{obj}
\end{equation}
subject to
\begin{subequations}
\begin{align}
\dot \xi(t) = f(\xi(t),\upsilon(t)), \quad &a.e. \  t \ \in [0,\,T]  \label{xi}\\
\xi(t) &\in M  \label{mani} \\
\xi(0) \in N_0,\ \ &\xi(T)\in N_f  \label{const}\\
\upsilon \in L^\infty([0,\,T],&U), \ U \subseteq \R^m. \label{control}
\end{align}
\end{subequations}
The final time $T$ can be fixed or variable, $M$ is a $n$-dimensional connected paracompact smooth manifold \(M\) and the state end points constraints $N_0,\, N_f$ are smooth connected embedded submanifolds of \(M\). We assume that $c_0,\,c_f,\,f^0,\,f$  are defined on open sets and that they are $C^\infty.$\\
  We take smooth time independent data because we are interested in the irregularities arising from the maximisation of the Hamiltonian.

The couple $(\xi,\,\upsilon)$ is called \emph{admissible} if it a solution of \eqref{xi} which satisfies the constraints \eqref{mani}--\eqref{const}--\eqref{control}, in this case we refer to $\xi$ as an \emph{admissible trajectory} and to $\upsilon$ as the \emph{associated control}.

We consider  as a \emph{candidate optimal solution} a given reference admissible trajectory  $\widehat\xi$ which is identified by the triple $(\widehat{T},\,\widehat{\xi},\, \widehat{\upsilon})$  and we study its strong local optimality according to the following Definition \ref{slm}.  
\begin{definition}\label{slm}[Strong local minimiser]\ \\
The reference admissible trajectory $\widehat\xi$ is a strong local minimiser of the above considered problem \eqref{obj}, if there are neighbourhoods  $\uc \subseteq \R \times M$ of the graph of $\widehat{\xi},$ denoted by $\Gamma_{\widehat{\xi}},$  and $\oc\subseteq \R \times M$ of $(\widehat{T},\,\widehat\xi(\widehat{T}))$ such that $\widehat\xi$
is a minimiser among all the admissible trajectories $\xi$  satisfying
\[
\Gamma_\xi \subset \uc, \quad (T,\,\xi(T)) \in \oc
\]
independently of the associated control. 
\end{definition}

When the final time is fixed this definition reduces to the usual definition of strong local minimum which uses the $C^0$ topology to describe a neighbourhood of the admissible trajectories; whereas in the case of variable end time our definition is  local with respect to the graph of $\widehat{\xi},$ and hence local with respect to both the final time and the final point.

This notion has been called {\em time-state--local optimality} in \cite{MR2092518}--\cite{MR2860348}, where it is used also a stronger version of optimality, called {\em state--local optimality}.

\section{Notations and preliminary results}
We recall that for every connected paracompact smooth manifold $M$ (see \cite{MR2954043}) there is a Riemannian structure which induces a distance $d_M$ on $M$ and the metric topology is the same as the original topology. This distance allows us to talk about Lipschitz property for maps on manifolds.

The next definition will be used to describe the main regularity assumption 
which is a strengthening of the usual Caratheodory-type assumption.

\begin{definition}\label{L-C}
Assume that $M,N$ are finite dimensional Riemannian smooth manifolds and $J$ is an open interval in $\R.$ We will say that the map
\(G :  J \times M \rightarrow N\)
is a Lipschitz--Carathéodory  map if it satisfies the following
\begin{enumerate}[i.]
\item
For almost every $t\in J$ the map $x \mapsto G(t,x)$ is locally Lipschitz.
\item
For each $x \in M$ the map $t \mapsto G(t,x)$  is bounded measurable.
\item
For any compact set $K \subseteq M$ there is an essentially bounded measurable function $m$ such that
\[
d_N(G(t,x),G(t,y)) \leq m(t)\, d_M(x,y), \ x,y \in K
\]
\end{enumerate}
\end{definition}
When $G$ is a time-dependent vector field and hence $N=TM$ is the tangent space to $M,$ this hypothesis assures existence, uniqueness and Lipschitz continuity in $(t,z_0)$ of the solutions of the differential equation
\[
\dot \zeta(t) = G(t, \zeta), \quad \zeta(t_0) = z_0, \ t_0 \in  J.
\]
Moreover we will use it in an analogous but simpler way when $N=\R$ to obtain that the Lipschitz continuity of $(t,z) \mapsto \int_{t_0}^t G(s,z)ds.$ 

Let $f$ be a vector field on the manifold $M$ and $\varphi : M \to \mathbb{R}$ be a smooth function. The action of $f$ on $\varphi$ (directional derivative or 
Lie derivative) evaluated at a point $x$ is denoted with one the two expressions
\[
L_f \varphi(x)= \langle d\varphi(x),f(x)\rangle.
\]
For any \(C^2\)-function \(\varphi\) such that 
\(d\varphi(x)=0\)  the second derivative \(D^2\varphi(x)\) is well defined as a bilinear symmetric form on \(T_x M\).
 
Finally we identify any bilinear form \(Q\) on a vector space \(V\) with a linear form \(Q:V \rightarrow V^*\) and we write
\[Q(v,w):=\langle Qv, w \rangle, \quad Q[v]^2:= Q(v,v,).\]
\begin{remark}
We will sometimes omit writing explicitly that equalities between $L^\infty$ functions hold almost everywhere unless we need to emphasise it.\\
\end{remark}
\subsection{Symplectic notations}
For a general introduction to symplectic geometry
and its application to variational problems we refer to \cite{MR2269239}
while for more specific applications to optimal control
we refer to \cite{MR524203}.

Denote by \( \pi : T^*M \rightarrow M\) the cotangent bundle; it is well known that $T^*M$ possesses a canonically defined 
symplectic structure, given by the symplectic form 
$\bsi_{\ell}=d\bs(\ell)$, where $\ell$ denotes an element of $T^*M$ and $\bs$ is the Liouville canonical 1-form $\bs(\ell)=\ell\circ \pi_*$. 

If $q_i$ are local coordinates on the base manifold $M$ and $p_i$ the fibre coordinates then in local coordinates we can write
\[
\bs :=\sum_{{\mathfrak i}=1}^n p_idq_i, \quad \bsi:=\sum_i {d}p_i \wedge {d}q_i,
\]
where $d$ denotes the exterior derivative and $\wedge$ denotes the exterior product.

We recall that the symplectic structure allows us to associate, with each time dependent locally defined Hamiltonian $H_t : T^*M \rightarrow R$  a unique vector field $\overset{\rightarrow}{H_t}$  on $T^*M$
defined by the action
\[
\langle d H_t(\ell), \cdot \rangle = \bsi_{\ell}(\cdot,\overset{\rightarrow }{H_t}(\ell)).
\]

This vector field defines a corresponding Hamiltonian system
\begin{equation}\label{arrow}
\dot \lambda(t) = \, \overset{\ \rightarrow}{H}_t(\lambda(t)), \quad \lambda(0)=\ell,\quad  a.e. \ t\ \in[0,{T}]
\end{equation}
and we will denote  its flow by $\ell \mapsto \hc(t, \ell)$. Note that for any time dependent object we will use the notation
\[
M(t, \ell) = M_t(\ell)=M_\ell(t)
\]
when one of the variables is to be considered as fixed.
\subsection{The Pontryagin maximum principle}
We assume that the candidate optimal solution $\widehat \xi$ is a state   extremal and with this we mean that it satisfies the PMP, see Assumption \ref{PMP}.\\ For simpler notations we set
\[
\widehat{x}_0:=\widehat\xi(0), \quad \widehat{x}_f:= \widehat\xi(\widehat T)
\]
To state the PMP we introduce three Hamiltonians; if $\uc$ is the common open domain of both $f$ and $f^0$ let $\Omega:= \pi^{-1}(\uc) \subseteq T^*M.$\\
We define
\begin{enumerate}[1)]
\item the (control dependent) \emph{pre-Hamiltonian} $F : \Omega \times U  \rightarrow \R$ 
as
\[F: (\ell, u) \mapsto \langle \ell , f(\pi(\ell),u)\rangle
 - p_0\,f^0( \pi(\ell),u),\]
\item the time dependent \emph{reference Hamiltonian}
\(\widehat F : [0,\,\widehat{T}] \times \Omega  \rightarrow \R\)
as
\[\widehat F : (t,\ell) \mapsto F(\ell,\up(t)),\]
\item the \emph{maximised Hamiltonian} $F_{\max} : \Omega  \rightarrow \overline \R$ as
\[
F_{\max}(\ell):=\sup_{u \in U} F(\ell,u).
\]
\end{enumerate}
Let us remark that the maximised Hamiltonian could take the value $+\infty.$\\
All the above Hamiltonians depend on the parameter \(p_0\) which can take the
values \(\{0,1\}\) characterising, respectively, the abnormal and normal case.
\begin{assumption}\label{PMP} [Pontryagin Maximum Principle] There is a non-trivial couple $(p_0,\widehat \lambda),$ where $p_0 \in \{0,1\}$ and $\widehat \lambda: [0,\,\widehat{T}] \rightarrow T^*M$ is a lifting
of $\widehat \xi$ to \(T^*M,\) (i.e. $\pi \circ \widehat \lambda = \widehat \xi$) satisfying the Hamiltonian system
\begin{equation}
\dot \lambda(t) = \, \overset{\rightarrow}{\widehat F_t}(\lambda(t)), \quad  a.e. \ t \in [0,\,\tf],
\label{pnt2}
\end{equation}
the transversality conditions
\begin{align}
\widehat \lambda(0) &= p_0\,dc_0(\widehat{x}_0) \qquad \text{on} \ T_{\widehat{x}_0}N_0 \label{trn1}\\
\widehat\lambda(\widehat{T}) &= -p_0\,dc_f(\widehat{x}_f) \quad \text{on} \ T_{\widehat{x}_f}N_f,  \label{trn2}
\end{align}
and the maximisation property
\begin{equation}
\widehat F_t(\widehat \lambda(t)) = \max_{u \in U} F(\widehat \lambda(t),u) = F_{\max}(\widehat \lambda(t)). \nonumber
\end{equation}
Moreover $\widehat F_t(\widehat \lambda(t)) $ is constant and it is zero when the end time $T$ is variable.
\end{assumption}
$\^\lambda$ is the {\em Pontryagin extremal} associated to the {\em state extremal} $\wh\xi$.
\noindent For simpler notations we set
\[\la:=\widehat \lambda(0),\quad \widehat{\ell}_f:=\widehat \lambda(\widehat{T}).\]
It is not difficult to see that the two functions $p_0\,c_0$ and $p_0\,c_f$ act on \(N_0\) and \(N_f\),
respectively, but they can be extended to an open set in \(M\) in such a way that the transversality conditions (\ref{trn1}) and (\ref{trn2}) hold on the whole tangent space. Namely we denote by $\alpha, \beta : M \rightarrow \R$ two locally defined functions
such that
\begin{align}
\alpha = p_0\,{c_0}\ \text{on} \ N_0, &\quad \la = d\alpha(\xa) \label{alpha}\\
\beta =p_0\,{c_f} \ \text{on} \ N_f, &\quad \lf =- d\beta(\xf). \label{beta}
\end{align}
The transversality conditions can also be expressed as
\begin{eqnarray*}
&\widehat{\ell}_0\in  \Lambda_0:= \{d\alpha(q)+\omega : q \in N_0,\ 
\omega \in T_q^\perp N_0 \}& \label{IM} \\
& \lf \in \Lambda_f:=\{-d\beta(q)+\omega : q \in N_f,\ 
\omega \in T_q^\perp N_f\}& \label{FM}
\end{eqnarray*}
where "\(\ ^\perp\)" means orthogonal with respect to the dual coupling.

In the normal case \(\alpha,\beta\) are cost functions equivalent to the original ones while in the abnormal case they are extensions of the zero function on the constraints.

When \(p_0=0\) all the costs disappear from our conditions and indeed we will study a problem with a zero cost.
Proving that \(\widehat \xi\) is a {\em strict} minimiser will imply that it is isolated among the admissible  trajectories.

\section{A Hamiltonian approach to optimality}
In this section we extend the known results for a $C^2$ maximised  Hamiltonian  in two different directions. We first prove that the Cartan form maintains its symplectic properties under weaker regularity assumptions on the Hamiltonian, Section \ref{cartan}. Afterwards, motivated by the singular case where the maximised Hamiltonian does not possess the required regularity properties, we introduce a super-Hamiltonian which can be used to compare the costs of admissible trajectories by lifting them to the cotangent bundle, Section \ref{super}. Using these results we prove abstract sufficient conditions, Section \ref{abstr}.

\subsection{The Cartan form} \label{cartan}
Let $\alpha : M \rightarrow \R$ be a smooth function and let  \(\Lambda \) be the graph of $d\alpha$ on a contractible open set. $\Lambda$ is a Lagrangian submanifold of \(T^*M\) and it is known that
\[
\bs_{|\Lambda} = d (\alpha \circ \pi).
\]
Let $\oc$ be an open set in $T^*M$ and $J \subseteq \R$ be an open interval, with  $J \supset [0,\,\widehat{T} ],$ and $\oc \supset \Lambda,$  we consider a time dependent Hamiltonian
$$ H : J \times \oc \subseteq \R \times T^*M \rightarrow \R.$$
Associated to this Hamiltonian we consider the Cartan form  on \(J \times \oc\)
\[
\bo = \bs\boldsymbol{-Hdt}.
\]
The following assumption describes a set of \emph{minimal} regularity conditions which the Hamiltonian $H$ has to guarantee to develop our approach. In specific examples the properties of $H$ will be different but they will imply this kind of regularity.

\begin{assumption}\label{hamiass}
Assume that
\begin{itemize}
\item[1.]
The flow 
\[(t,\ell) \in J \times \Lambda \mapsto \hc(t,\ell) \in T^*M \] is well defined and Lipschitz continuous.
\item[2.]
The function
\[
(t,\ell) \in J \times \Lambda \mapsto \langle \hc_t(\ell),\, \pi_* \overset{\rightarrow}{H}_t \circ \hc_t(\ell) \rangle - H_t \circ \hc_t(\ell) \in \R
\]
is Lipschitz Caratheodory (Definition \ref{L-C})
\end{itemize}
\end{assumption}

Thanks to Assumption \ref{hamiass} the Cartan form defines a Lipschitz function  $\theta : J \times \Lambda \rightarrow \R$ which will play a crucial role. Let
\begin{align} \label{theta}
\theta (t,\ell) := (\alpha \circ \pi)(\ell)&+\int_{{\hc_\ell}_{|[0,\,t]}}\bo=\nonumber\\
=(\alpha \circ \pi)(\ell)&+\int_0^t \langle \hc_s(\ell),\, \pi_* \overset{\rightarrow}{H}_s \circ \hc_s(\ell) \rangle - H_s \circ \hc_s(\ell)\rangle ds 
\end{align}
and hence
\begin{equation}\label{dtheta}
\partial_t\theta (t,\,\ell) = \langle \hc_t(\ell),\, \pi_* \overset{\rightarrow}{H}_t \circ \hc_t(\ell) \rangle - H_t \circ \hc_t(\ell)\quad a.e \ t \in J
\end{equation}
Let us now prove that, also in this case, the classical property of the Cartan form holds, that is the form $\hc^*\bo$ is exact on $J \times \Lambda.$ 

We will consider differential 1-forms with possibly $L^\infty$ coefficients; these forms are called (Whitney) flat forms and their  properties are presented in \cite{MR2177410}. The general theory is fully presented in Whitney's monograph \cite{MR0087148}, see also \cite{MR0257325}. We do not use the results in their full extent, what we really need is the chain rule for Lipschitz functions as it can be found in Lemma 4.6.3 in \cite{MR1859913} and a suitable version of the Stokes Theorem; an explicit proof is in  \cite{MR1819736} but the essential elements needed for the proof are contained in the cited books,  \cite{MR0257325} for the general case or \cite{MR2177410} for the Lipschitz one. \\
The following lemma is the first step to prove that the form $\hc^*\bo$ is exact on $J \times \Lambda.$
\begin{lemma}
For every given $t \in J$ the form $\hc_t^*\bs $ is exact on $\Lambda$ and 
\begin{equation}\label{exact1}
\hc_t^*\bs(\ell) = d\, \theta_t(\ell) \quad  a.e.\  \ell \in \Lambda 
\end{equation}
\end{lemma}

\begin{proof}
It is equivalent to prove that for every Lipschitz curve  $\gamma :[0,\, 1] \rightarrow \Lambda$ 
\[
\int_\gamma \hc_t^*\bs = \int_\gamma d\, \theta_t =\theta_t(\gamma(1)) -\theta_t(\gamma(0)).
\]
Let us consider the set
\[\Delta:=\Bigl\{(\tau,s) :
0\leq\tau \leq t ,\, 0 \leq s  \leq 1\Bigr\}
\]
and the map \[\phi : (\tau,s) \mapsto
\hc(\tau,\gamma(s)).\]
by the Stokes theorem we have
\begin{align*}
\int_\Delta \phi^* \bsi =&
\int_{\phi_{|s=0}}\!\!\!\bs + \int_{\phi_{|\tau=t}}\!\!\!\bs - \int_{\phi_{|s=1}}\!\!\!\bs - \int_{\phi_{|\tau=0}}\!\!\!\bs =\\
 = & \int_{{\hc_{\gamma(0)}}_{|[0,\,t]}}\!\!\!\bs + \int_{\hc_{t}\circ \gamma_{|[0,\,1]}} \!\!\!\bs - 
\int_{{\hc_{\gamma(1)}}_{|[0,\,t]}}\!\!\!\bs - \int_{\gamma_{|[0,\,1]}}\bs 
\end{align*}
moreover
\begin{align*}
\int_\Delta \phi^* \bsi =&
\int_0^t \Bigl\{\int_0^1 \bsi \left(\overset{\ \rightarrow}{H}_\tau \circ \hc_\tau \circ \gamma(s), \partial_s\,\phi(\tau,\,s) \right)\,ds \Bigr\} d\tau= \\
=&
-\int_0^t \Bigl\{\int_0^1 \langle dH_\tau\circ \phi(\tau,\,s),\, \partial_s\,\phi(\tau,\,s) \rangle \,ds \Bigr\} d\tau =\\
=& -\int_0^t H_\tau\circ \hc_{\gamma(1)}\, d\tau + \int_0 ^t H_\tau\circ\hc_{\gamma(0)}\,d \tau
\end{align*}
by equating the two right hand sides we obtain
\begin{align*}
\int_{\hc_{t}\circ \gamma_{|[0,\,1]}} \bs&=
 \int_{\gamma_{|[0,\,1]}}\bs +\int_{{\hc_{\gamma(1)}}_{|[0,\,t]}}\!\!\!\!\!\bo
-\int_{{\hc_{\gamma(0)}}_{|[0,\,t]}}\!\!\!\!\!\bo=\\
 &= \theta(t,\,\gamma(1)) -\theta(t,\,\gamma(0)).
\end{align*}

\end{proof}
We are now able to prove the main result of this section
\begin{lemma}\label{exact2}
The differential form $\hc^*\bo $ is exact on \(J\times \Lambda\) and
\begin{equation*}
\hc^*\bo = d\, \theta \quad \ a.e.\ (t,\ell) \in J \times \Lambda
\end{equation*}
\end{lemma}
\begin{proof}
We can equivalently prove that for every Lipschitz curve $$\mu : t \in [0,\,1] \mapsto (t,\gamma(t)) \in  J\times \Lambda$$ 
we have
\begin{equation*}
\int_\mu \hc^*\bo = \int_\mu d\,\theta = \theta(1,\gamma(1))-\theta(0,\gamma(0)).
\end{equation*}
By definition we have that
\begin{align*}
\int_\mu \hc^*\bo &= \int_\mu \hc^*\bs - \int_0^1 H\left(s, \hc_{s}\circ \gamma(s)\right) ds =\\
=& \int_0^1 \langle \hc_{s}\ \circ \gamma(s),\, \frac{d}{ds} \pi\hc_{s} \circ \gamma(s)\rangle ds- \int_0^1 H\left(s, \hc_{s}\circ \gamma(s)\right) ds=\\
= & \int_0^1 \left\{ \langle \hc_{s}\ \circ \gamma(s),\, \pi_*\overset{\ \rightarrow}{H}_{s}\circ \hc_{s}\ \circ \gamma(s)\rangle -H\left(s, \hc_{s}\circ \gamma(s)\right)\right\}ds +\\
+& \int_0^1 \langle \hc_{s}\ \circ \gamma(s),\, \pi_*\hc_{s*}  \dot \gamma(s)\rangle ds 
\end{align*}
and by \eqref{dtheta} and \eqref{exact1}
\begin{align*}
\int_\mu \hc^*\bo = &  \int_\mu \partial_s \theta (s,\ell)ds + \int_\mu  d\theta_{s}(\ell) =\\
=& \int_\mu d \theta (s,\ell)= \theta(1,\gamma(1))-\theta(0,\gamma(0)).
\end{align*}
\end{proof}

When the state projection of the Hamiltonian flow is invertible we obtain some important properties of the function $\theta$. We recall  that a homeomorphism between metric spaces is said to be  bi-Lipschitz if it is Lipschitz and has a Lipschitz inverse.

\begin{lemma}\label{dtheta2}
Suppose that, for a given $t \in J,$ the function $\pi\hc_t$ is bi-Lipschitz  then on  $\pi \hc_t(\Lambda)$
\begin{enumerate}[(i)]
\item
$
d \left( \theta_t \circ \hi \right) = \hc_t \circ \hi. 
$
\item The function $\alpha_t :=  \theta_t \circ \hi$ is a $C^1$ function.
\item $\Lambda_t := \hc_t(\Lambda)$ is the graph of $d\alpha_t.$
\end{enumerate}
\end{lemma}
\begin{proof}\ \\
Since  $ \theta_t \circ \hi $ is Lipschitz then by the chain rule we can prove (i) by  proving that their integrals coincide over any Lipschitz curve  $\gamma : [0,\,1] \rightarrow \pi\hc_t(\Lambda).$ \\
By  \eqref{exact1} we have
\begin{align*}
\int_\gamma d\left(\theta_t \circ \hi \right)= \ &\int_0^1 \langle  d\left(\theta_t \circ \hi\right)(\gamma(s)),\, \dot  \gamma(s) \rangle\, ds =\\
= \ & \int_0^1 \langle  d\theta_t \circ \hi\bigl(\gamma(s)\bigr),\, \hi_*\dot  \gamma(s) \rangle\, ds =\\
= \ & \int_{\hi \circ \gamma}d\,\theta_t = \int_{\hi \circ \gamma} \hc_t^*\bs=\\
= \ & \int_0^1 \langle \hc_t \circ \hi\left(\gamma(s)\right),\, \frac{d}{ds}\left(\pi \hc_t \circ \hi \circ \gamma(s)\right)\rangle ds=\\
= \ & \int_0^1 \langle \hc_t \circ \hi \left(\gamma(s) \right),\,\dot \gamma(s) \rangle ds
\end{align*}
which proves statement (i). (ii) Follows immediately from (i). 
To prove (iii) it is sufficient to notice that
\[
\Lambda_t = \hc_t \circ \hi(\pi \hc_t(\Lambda)).
\]
\end{proof}

\subsection{The super Hamiltonian and its properties.} \label{super}
Throughout this section we assume that $\alpha$ satisfies \eqref{alpha} and   that \(\Lambda \) is the graph of $d\alpha$ on a contractible neighbourhood of $\widehat{x}_0.$ Moreover we assume that $H$ satisfies  Assumption \ref{hamiass} and the following Assumption \ref{superh}, this last motivates the name of super-Hamiltonian. We underline that these assumptions concern jointly the super-Hamiltonian and the horizontal Lagrangian manifold $\Lambda.$
\begin{assumption}\label{superh}
The Hamiltonian $H$ satisfies the following
\begin{subequations}
\begin{align}
&H_t \circ \hc_t(\ell) \geq F_{\max}(\hc_t(\ell)), \ \ell \in \Lambda \label{shami1}\\
&H_t \circ \widehat \lambda(t)=\widehat F_t \circ \widehat \lambda(t)  = F_{\max}\circ \widehat \lambda(t), \ a.e.  \ t \in  J\label{shami3}\\
&\overset{\ \rightarrow}{H_t}(\widehat \lambda(t))=\overset{\ \rightarrow}{\widehat F_t}(\widehat \lambda(t)), \ a.e.\  t \in  J. \label{shami4}
\end{align}
\end{subequations}
\end{assumption}
\noindent As a consequence  the Pontryagin extremal \(\widehat \lambda\) is also a solution of the system
\begin{equation}
\dot \lambda(t) = \, \overset{\rightarrow}{H_t}(\lambda(t))\quad  a.e. \ t\ \in[0,\wh{T}]. \label{hami}
\end{equation}
\begin{remark}
If the maximised Hamiltonian $F_{\max}$ satisfies Assumption \ref{hamiass} then it satisfies also Assumption \ref{superh}; for example in the classical case when $F_{\max}$ is $C^2$ we can take $H:=F_{\max}$ for any $\Lambda.$
\end{remark}
From Assumption \ref{hamiass} it follows that  there exists an interval $I:=[0,\,{T}]$ with $ T > \widehat{T}$ such that $\hc_t(\la)$ is defined for $t \in I$ and we have ${\hc_t(\la) = \widehat{\lambda}(t)}$ on $[0,\, \widehat{T}].$ Moreover from the compactness of the time interval $I$, it follows that there is an open neighbourhood \(\oc_\la\) of \(\la\) such that,  without loss of generality, we can redefine $\Lambda := \Lambda \cap \oc_\la,$ to obtain that \(\hc\) is defined on \(I \times \Lambda\) and \(\fc\) is defined on \([0,\,\widehat{T}] \times \Lambda,\) where 
 $\fc$ is the flow of $\overset{\ \rightarrow}{\widehat F_t}.$ 
 
{ To prove the main Theorem we need a final crucial assumption concerning  the invertibility of the projection of the flow of the Hamiltonian onto the state space.}
\begin{assumption}\label{lipo}
Assume that the function
\begin{align*}
id_{I} \times \pi\hc: (t,\ell) \in I \times \Lambda &\mapsto (t, \pi \hc_t(\ell))  \in I \times M
\end{align*}
is bi-Lipschitz  between  $I \times \Lambda$ and an open set $\uc$ of $ I \times M$  containing the graph of $\widehat{\xi}.$ 
\end{assumption}
To verify the bi-Lipschitz assumption one can use one the the available inverse function theorems for locally Lipschitz functions which are based on the local invertibility properties of $\pi_{*}\hc_{t*}$ in $[0,\,\widehat{T}] \times \{\la \}. $ We refer, for example, to the one proposed by F. Clarke in \cite{MR0425047}.
\begin{remark}
For optimal control problems where the initial point is free the problem is always normal, moreover the initial Lagrangian manifold $\Lambda_0$ is itself horizontal and we take it as  $\Lambda.$ In general, however, this is not the case and one has to define an appropriate $\alpha,$ in our approach we expect that the suitable $\alpha$ can be obtained from the second order conditions as it is the case in the applications we describe in the final Section \ref{appl}.
\end{remark}
We can now state the main Theorem which compares the reference cost with the cost of an admissible neighbouring trajectory.
\begin{theorem}\label{main1}
Under Assumption \ref{hamiass}--\ref{superh}--\ref{lipo} , let $\xi : [0,T]  \rightarrow M$ 
be an admissible trajectory whose graph is contained in $\uc$ and let 
$\rho(t) := \left(\pi\hc_t\right)^{-1} \circ \xi(t) $
then
\begin{equation}\label{cost}
p_0\left(J({T},\xi,\upsilon)-J(\widehat{T},\widehat{\xi},\up)\right) \geq \beta(\xi({T})) +\theta({T},\,\rho({T}))- \beta(\xf)-\theta(\widehat{T}, \widehat{\ell}_0)
\end{equation}
\end{theorem}

\begin{proof}
Let $\mu := \hc_t \circ \rho$ be the lift of $\xi$ to $T^*M,$ by Lemma \ref{exact2} we obtain
\begin{align*}
0 &= \int _0^{T}\left\{\langle \mu(t),\, \dot \xi(t) \rangle - H_t \circ \mu(t)\right\} dt+\theta(\widehat{T}, \widehat{\ell}_0) -\theta(T,\,\rho(T))+\\
& \quad - \int _0^{\widehat{T}}\left\{\langle \widehat{\lambda}(t),\,  \widehat{f}_t \circ \widehat{\xi}(t)\rangle - \widehat{F}_t \circ \widehat{\lambda}(t)\right\} dt + \alpha (\pi(\mu(0))) -\alpha (\pi(\widehat{\ell}_0))=\\
&=\int _0^{T}\left\{\langle \mu(t),\, \dot \xi(t) \rangle - H_t \circ \mu(t)\right\} dt+\theta(\widehat{T}, \widehat{\ell}_0) -\theta({T},\,\rho({T}))+\\
&\quad  +\int_0^{T} p_0\,f^0(\xi(t), \upsilon(t))dt-\int_0^{T} p_0\,f^0(\xi(t), \upsilon(t))dt +\\
&\quad - \int _0^{\widehat{T}}\left\{\langle \widehat{\lambda}(t),\,  \widehat{f}_t \circ \widehat{\xi}(t)\rangle - \widehat{F}_t \circ \widehat{\lambda}(t)\right\} dt + \alpha (\xi(0)) -\alpha (\xa)+\\
& \quad +\int_0^{\widehat{T}} p_0\,f^0(\widehat{\xi}(t), \up(t))dt-\int_0^{\widehat{T}} p_0\,f^0(\widehat{\xi}(t), \up(t))dt.
\end{align*}
If we isolate, on the left hand side, the terms which describe the costs of the two trajectories we obtain
\begin{align*}
\alpha (\xi(0)) &+\int_0^{T} p_0f^0(\xi(t), \upsilon(t))dt -\alpha (\widehat{\xi}(0))-\int_0^{\widehat{T}} p_0f^0(\widehat{\xi}(t), \up(t))dt =\\
 & =-\int _0^{T}\left\{\langle \mu(t),\, f(\xi(t), \upsilon(t))\rangle -p_0\,f^0(\xi(t), \upsilon(t))- H_t \circ \mu(t)\right\} dt+\\
&\quad +\int _0^{\widehat{T}}\left\{\langle \widehat{\lambda}(t),\, {f} (\widehat{\xi}(t), \up(t))\rangle -p_0f^0(\widehat{\xi}(t), \up(t)) - \widehat{F}_t \circ \widehat{\lambda}(t)\right\} dt+\\
&\quad -\theta(\widehat{T}, \widehat{\ell}_0) +\theta({T},\,\rho({T}))=\\
&=-\int _0^{T}\left\{F( \mu(t),\, \upsilon(t))- H_t \circ \mu(t)\right\} dt-\theta(\widehat{T}, \widehat{\ell}_0) +\theta({T},\,\rho({T}))
\end{align*}
where we have used the definition of $F$ and of $\widehat{F}.$ 
Now by the \eqref{shami1} property of the super Hamiltonian, by adding to both sides $\beta(\xi({T}))-\beta(\xf),$ and recalling that on the initial manifold $\alpha =p_0\,c_0$ and on the final manifold $\beta = p_0\,c_f$  we obtain
\begin{equation*}
p_0\left(J({T},\xi,\upsilon)-J(\widehat{T},\widehat{\xi},\up)\right) \geq \beta(\xi({T})) +\theta({T},\,\rho({T}))- \beta(\xf)-\theta(\widehat{T}, \widehat{\ell}_0)
\end{equation*}
\end{proof}

\subsection{Abstract sufficient optimality conditions.} \label{abstr}
We are now able to state an abstract theorem which reduces the strong local optimality of an admissible reference trajectory $\widehat{\xi}$ to the local optimality of a suitable function of the right end point $(\widehat{T},\,\xf)$ of the graph of $\widehat{\xi}.$ \\
Let  \[
\Phi \colon (t,\,x)\in \uc \mapsto \beta(x)+ \theta(t,\,\hi(x)) = (\beta + \alpha_t)(x) \in  \R, 
\]
where $\alpha_t =\hc_t \circ \hi$ was defined in part (ii) of Lemma \ref{dtheta2}.
By means of this function we can rewrite the inequality \eqref{cost} as 
\begin{equation}\label{cost2}
p_0\left(J({T},\xi,\upsilon)-J(\widehat{T},\widehat{\xi},\up)\right)\ \geq \ \Phi({T}, \xi({T})) -\Phi(\widehat{T},\,\xf).
\end{equation}
This relation will allows us to characterise the minima of the function $p_0J$ by studying the function $\Phi$ at the reference point $(\widehat{T},\,\xf)$. Indeed we can now state a sufficient optimality condition
\begin{theorem}\label{suff}
In the normal case, $p_0=1$, under Assumptions \ref{hamiass}--\ref{superh}--\ref{lipo}, if there exists  a neighbourhood $\oc$ of ${(\widehat{T},\,\xf)}$ in $N_f$  
such that $(\widehat{T},\xf)$ is a minimiser of $\Phi$ restricted to  $\oc$  then $(\widehat{T},\,\widehat{\xi},\up)$ is a strong local minimiser.
\end{theorem}
\begin{proof}
It follows immediately from \eqref{cost2}.
\end{proof} 
\begin{remark}
When the initial point is free we have that $ \alpha = c_0$, if, moreover, the final point and time are fixed then the right hand side of \eqref{cost2} is zero and we can prove that $\widehat{\xi}$ is a strong local minimiser by only verifying Assumptions \ref{hamiass}--\ref{superh}--\ref{lipo}
for the given $\Lambda$. In the other cases we have to find a suitable $\Lambda$ and to prove the minimality property of $\Phi$ at $(\widehat{T},\xf).$
\end{remark}

\begin{remark}
Concerning the abnormal case, $p_0=0,$ if one can prove that the minimum is strict then, as a byproduct, it follows that $\widehat{\xi}$ is isolated among the admissible trajectories. To prove that  $\widehat{\xi}$  is a strict strong local minimiser the first step is to require that  ${(\widehat{T},\,\xf)}$ is a strict local minimiser for $\Phi,$ in this case, if we have another minimiser $(T, \xi, \upsilon),$ we can conclude that $T=\widehat{T}$ and $ \xi(T)=\xf.$ From the proof of Theorem \ref{main1} we have
\begin{equation*}
0=p_0\left(J({T},\xi,\upsilon)-J(\widehat{T},\widehat{\xi},\up)\right) =-\int _0^{\widehat{T}}\left\{F( \mu(t),\, \upsilon(t))- H_t \circ \mu(t)\right\} dt.
\end{equation*}
By Assumption \eqref{shami1} we obtain that
\begin{equation*}
F( \mu(t),\, \upsilon(t))- H_t \circ \mu(t)=0, \quad a.e. \ t \in [0,\,\tf].
\end{equation*}
On the other hand, since $\mu = \hc_t \circ \hi \circ \xi$ we have that 
\begin{equation*}
\dot \mu(t) = \overset{\rightarrow} H_t(\mu(t))- \hc_{t*}\hi_{*}\left(\pi_* \overset{\rightarrow}H_t (\mu(t))-f(\xi(t), \mu(t))\right)\  a.e. \ t\ \in[0,\tf].
\end{equation*}
If 
\begin{equation}\label{zero}
\pi_* \overset{\rightarrow}H_t (\mu(t))-f(\xi(t), \mu(t))=0{\quad  a.e. \ t\ \in[0,\tf]}
\end{equation}
then one can conclude that $\mu = \widehat{\lambda}$ since they both satisfy the same Hamiltonian equation with the same boundary conditions. Unfortunately Assumption \eqref{shami1} is too mild to prove \eqref{zero}. 
We can strengthen it by assuming that
\begin{equation}
H_t (\ell) \geq F_{\max}(\ell), \ \ell \in \oc 
\end{equation}
where $\oc$ is a neighbourhood of the range of $\widehat{\lambda}; $ we note that this is true if $H_t=F_{\max}.$ From this new assumption we deduce that
\begin{equation}\label{zero2}
F( \ell, \upsilon)- H_t(\ell)=0, \quad  \forall \ \ell \ \colon \ \pi \ell = \xi(t), {\quad a.e. \ t \in [0,\,\tf]},
\end{equation}
hence by  \eqref{zero2}  
\begin{equation}
0= \partial_v F( \mu(t), \upsilon)-\partial_v H_t(\mu(t)) = \pi_* \overset{\rightarrow}H_t (\mu(t))-f(\xi(t), \mu(t)),
\end{equation}
where $\partial_v$ is the vertical derivative along the fibre.
This is not the only way to obtain local uniqueness of the strong minimiser; for the case of singular control see for example \cite{MR2860348}, \cite{ChiSte2015}.
\end{remark}
To state necessary and/or sufficient condition for $(\tf,\xf)$ to be a local minimiser for  $\Phi$  we compute its first and second derivatives. We note that, by Lemma \ref{dtheta2}, the function $\Phi$ is a $C^1$ function of $x$ at $t$ fixed and  Lemma \ref{phi} states that the point $(\widehat{T},\,\xf)$ is a critical point for the function $\Phi$ without further assumptions on the data. On the other hand to  obtain the existence of second derivatives  we will require stronger regularity assumptions on the data which here we do not specify. We underline that since the first derivatives are zero, then the second ones are well defined as a quadratic function on $\R\times T_{\xf}M$
\begin{lemma}\label{phi}
Under Assumptions \ref{hamiass}--\ref{superh}--\ref{lipo}, we have that
\begin{enumerate}[(i)]
\item $\displaystyle{\partial_x \Phi(t,x) = d \beta(x) + \hc_t \circ \hi(x)}= d(\beta+\alpha_t)(x), \ t \in I$ 
\item $\partial_t \Phi(t,x) =  - H_t\circ \hc_t \circ \hi (x) =  - H_t\circ (d\alpha_t) (x), \ a.e. \ t \in I.$
\end{enumerate}
Moreover
\[
d\,\Phi(\widehat{T},\,\xf)=0.
\]
\end{lemma}

\begin{proof}
(i) follows immediately from Lemma \eqref{dtheta2}. To prove (ii) from \eqref{dtheta} and Lemma \eqref{dtheta2} we have
\begin{align*}
\partial_t \Phi(t,x) = & \ \partial_t \theta(t,\,\hi(x))- \langle d\left(\theta_t \circ  \hi\right)(x), \, \pi_*\overset{\rightarrow}{H}_t \circ \hc_t \circ \hi(x)\rangle =\\
= & \ \partial_t \theta(t,\,\left(\pi \hc_t\right)^{-1}(x))- \langle \hc_t \circ  \hi(x), \, \pi_*\overset{\rightarrow}{H}_t \circ \hc_t \circ \hi(x)\rangle =\\
= &\ - H_t\circ \hc_t \circ \hi (x).
\end{align*}
Computing these derivatives at $(\widehat{T}, \xf),$ by the transversality condition \eqref{trn2},  we obtain 
\begin{align*}
\partial_x \Phi(\widehat{T}, \xf) = &\ d\beta(\xf)+ \hc_{\widehat{T}} \circ \left(\pi \hc_{\widehat{T}}\right)^{-1}(\xf )=\\
= & \ d\beta(\xf)+\lf = 0.
\end{align*}
By Assumption \eqref{shami3} and by the PMP with variable final time we have
\[
\partial_t \Phi(\widehat{T}, \xf)  = -H_{\widehat{T}}(\lf)=-\widehat{F}_{\widehat{T}}(\lf)= 0.
\]
When the final time is fixed we are not interested in this derivative.
\end{proof}
\begin{lemma}\label{2deriv}
Let Assumptions \ref{hamiass}--\ref{superh}--\ref{lipo} hold true and assume moreover that $\Phi$ 
is $C^2$ in a neighbourhood $\oc$ of  $(\tf,\xf)$, then
\begin{enumerate}
\item
$\displaystyle
\partial_{xx} \Phi (\tf,\xf)\,[\delta x]^2 =  \ D^2 (\beta + \alpha_\tf) (\xf) [\delta x]^2 =
\bsi \left( (d\alpha_\tf)_* \delta x,  d(-\beta)_*  \delta x \right)$
\item
$\displaystyle
\partial_{tx} \Phi (\tf,\xf)\,\delta x = -L_{\delta x} L_{\widehat{f}_\tf} \alpha_\tf(\xf) =
\bsi \left( \overset{\rightarrow }{H}_\tf(\lf),  (d\alpha_\tf)_* \delta x \right)$
\item
$\displaystyle
\partial_{tt} \Phi (\tf,\xf) =   -\partial_t H(\tf,\xf) - L^2_{\widehat{f}_\tf} \alpha_\tf(\xf) =\\
=-\left(\partial_t H(\tf,\xf) +
\bsi \left((d\alpha_\tf)_*\wh f_\tf, \overset{\rightarrow }{H}_\tf(\lf) \right)\right)$.
\end{enumerate}
Notice that since  $\alpha_\tf =\hc_\tf \circ \hif$ then $(d\alpha_\tf)_*=\hc_{{\tf\, *}} \hif_*$.
\end{lemma} 
\begin{proof}\ \\
{\em 1.} It follows directly from Lemma \ref{phi}, part (i).\\
{\em 2.-3.} 
From Lemma \ref{phi}, part (ii), taking into account that $\overset{\rightarrow }{H}_\tf(\lf)=\overset{\rightarrow }{\wh F}_\tf(\lf) $
and some known symplectic equalities, that can be easily obtained
in coordinates, we get
\begin{align*}
\partial_{tx} \Phi (\tf,\xf)\,\delta x = &\  
-\bsi\left((d\alpha_\tf)_*\delta x, \overset{\rightarrow }{H}_\tf(\lf) \right)=\\
= &\ -L_{\delta x} L_{\widehat{f}_\tf} \alpha_\tf(\xf)
\end{align*}
and
\begin{align*}
\partial_{tt} \Phi (\tf,\xf) = & \ \partial_t\Bigl(-H(t,\hc_t\hi(\xf))\bigr)=\\
= & -\partial_t H(\tf,\xf)+\bsi\left(\hc_{{\tf\, *}} \hif_*\pi_* \overset{\rightarrow }{H}_\tf(\lf), \overset{\rightarrow }{H}_\tf(\lf) \right)=\\
= & -\partial_t H(\tf,\xf)+\bsi\left((d\alpha_\tf)_*\wh f_\tf(\xf), \overset{\rightarrow }{H}_\tf(\lf) \right)=\\
= &  -\partial_t H(\tf,\xf) + L^2_{\widehat{f}_\tf} \alpha_\tf(\xf).
\end{align*}
\end{proof}

\subsection{The Minimum Time Problem}
In this section we apply the results of the previous sections to the minimum time problem which is a special case which can be obtained from the general one by setting
\[
c_0=c_f=0, \  f^0=1.
\]
For the minimum time problem we obtain a sufficient condition which do not involve the second derivatives of $\Phi$, but requires more regularity of the super-Hamiltonian "near the final time" and a second assumption which is clearly verified when final point is fixed.
\begin{theorem}\label{main2}
Let Assumptions \ref{hamiass}--\ref{superh}--\ref{lipo} hold true and assume moreover that
\begin{enumerate}
\item
The function  $(t,x) \mapsto H_t \circ \hi(x)$ is Lipschitz in a neighbourhood 
of $(\widehat{T},\,\widehat{x}_f) $ in $\R\times M$, with a Lipschitz constant given by $L >0$.
\item
$\displaystyle
\hc_t \circ \hi(x) \in T_x^{\perp}N_f, \quad (t,x) \in \oc_{(\widehat{T},\,\widehat{x}_f)}, \quad x\in N_f
$
\end{enumerate}
if $p_0=1$ (normal case),
then $({\widehat{T}},\,\widehat{\xi})$ 
is a strong local minimum for the minimum time problem.
\end{theorem}
\begin{proof}
Let k be the dimension of the submanifod
$N_f$,   we can take local coordinates centered at $\widehat{x}_f$ such that $N_f$ is homomorphic to the plane
given by the last $n-k$ coordinates equal to zero,
so that we can consider the problem on $\R\times\R^n$ with $\xf=0 $ and $\xi(T)\in \R^k$.

For the minimum time problem in the normal case equation
 \eqref{cost2}  reads
\begin{equation}\label{minTime}
T-\widehat{T} \geq \Phi(T, \xi(T)) -\Phi(\widehat{T},\,\xf) = \theta(T, \hi(\xi(T))-\theta(\widehat{T},\widehat{x}_f).
\end{equation}
By contradiction assume there exist admissible trajectories $\xi_n: [0,t_n]\to M,\ t_n<\tf $, whose graph is contained in $\uc$ 
and such that 
\begin{equation}\label{tozero}
\tf -t_n\to 0\, ,\quad\mbox{and}\quad \| \xi_n(t_n)\|\to 0 .
\end{equation}
Consider the curve
\[
\gamma := t \in [t_n,\,\widehat{T}] \mapsto \frac{\widehat{T}-t}{\widehat{T}-t_n} \xi_n(t_n) 
\]
which is such that $\gamma(t_n)= \xi_n(t_n)$ and $ \gamma(\widehat{T}) = 0.$ 
By Lemma \ref{exact2} and Assumption (ii)  we get from equation \eqref{minTime}
\begin{align*}
t_n - \tf & \geq  \int^{\widehat{T}}_{t_n}
\left\{\langle \hc_t \circ \hi(\gamma(t)),\,\dot \gamma(t)\rangle- H_t \circ \hi (\gamma(t))\right\}dt=\\
& = - \int^{\widehat{T}}_{t_n} H_t \circ \hi (\gamma(t))dt
\end{align*}
Since  $\left(\pi \hc_{\widehat{T}}\right)^{-1}(\widehat{x}_f)=\widehat{\ell}_0$ by \eqref{shami3} and Assumption (i)
we can write
\begin{align*}
\tf-t_n &\leq \int^{\widehat{T}}_{t_n}   H_t \circ \hi(\gamma(t))dt  =\\
&=  \int^{\widehat{T}}_{t_n}\left( H_t \circ \hi(\gamma(t))-H_{\widehat{T}} \circ \left(\pi \hc_{\widehat{T}}\right)^{-1}(\widehat{x}_f)
\right) dt \ \leq\\
& \leq \int^{\widehat{T}}_{t_n} \left \|H_t \circ \hi(\gamma(t)) -
H_{\widehat{T}} \circ \left(\pi \hc_{\widehat{T}}\right)^{-1}(\widehat{x}_f)\right\| dt \leq\\
&\leq \int^{\widehat{T}}_{t_n} L\,\left (\tf - t+\frac{\widehat{T}-t}{\widehat{T}-t_n}\| \xi_n(t_n) \|\right)dt =\\
& = L\,\frac{(\tf-t_n)^2}{2}\left(1+ \frac{\| \xi_n(t_n) \|}{\widehat{T}-t_n} \right) .
\end{align*}
Dividing by $\widehat{T}- t_n >0$ we obtain
\[
1\leq \frac{L}{2}\,(\tf - t_n+\|\xi_n(t_n)\|)
\]
which yield a contradiction by \eqref{tozero}.
\end{proof}

\section{Final Comments}\label{appl}
This general unified Hamiltonian approach can be used in different situations which has been already addressed in some published papers, where second order conditions have also been developed and appropriate $H$ and $\Lambda$ have been defined to obtain sufficient conditions for strong local optimality, here we briefly summarise them.

\begin{itemize}
\item  {\bf Bang--bang control.}\\ In this case $H:=F_{\max}$ is continuous and $\hc$ piecewise $C^\infty$ and Lipschitz. The case when there are a finite number of simple switches is studied in \cite{MR1972500} for a Mayer problem with variable end--points on a fixed time interval and in \cite{MR2201070} for a Bolza problem, while the corresponding minimum time problem is studied in \cite{MR2092518}. The double switch case is addressed in \cite{MR2765660} for a Mayer problem and in \cite{PoSpa2015} for the minimum time problem. The numerical analysis with Maple of a case study is in 
\cite{MR2082966}. 

\item {\bf Totally singular control}\\In \cite{MR2487762}, for a Mayer problem and a single-input system the author first introduced a super--Hamiltonian, this step was essential because the maximised Hamiltonian does not define a regular flow. In \cite {MR2860348} and \cite{MR3205110} the result for the minimum time problem has been obtained as a by--product of problems with controls containing bang and singular arcs.  Furthermore in \cite{MR2977718} the minimum time problem has been studied in the multi--input case for a system where the controlled vector fields  generate an involutive Lie algebra, while in \cite{ChiSte2015}  this last assumption is removed. We notice that the super Hamiltonian, and hence its flow, is $C^\infty$ in a neighbourhood of the range of $\wh \lambda$ but it satisfies Assumption \ref{superh} only starting from  a submanifold of $T^*M$ containing $\Lambda,$ as we require.
\item {\bf Bang--singular control}\\ The last application we consider is the case when the reference trajectory contains both singular and bang arcs.
The obtained results concern the minimum time problem for a single-input control system. In \cite{MR2860348} a
bang-singular-bang trajectory in a problem with fixed end points is addressed, while in \cite{MR3205110} it is considered a bang-singular trajectory with the initial point fixed and the final one constrained to the integral line of the controlled vector field. In this case the flow of the super Hamiltonian is sufficiently regular and satisfies Assumption \ref{superh}  only starting from  $\Lambda.$ 

\end{itemize}

\def\cprime{$'$}

\end{document}